\documentclass[11pt,a4paper]{article}

\usepackage[margin=1in]{geometry}
\usepackage{xcolor}
\usepackage[T1]{fontenc}
\usepackage[utf8]{inputenc}
\usepackage[english]{babel}
\usepackage{graphicx}
\usepackage{amsfonts,amsmath,amssymb,amsthm}

\usepackage{tikz}
\usepackage{pgfplots}
\pgfplotsset{compat=1.18}

\usepackage{enumitem}
\usepackage{wrapfig}

\newcommand{\cF}{\mathcal{F}}

\newcommand{\C}{\mathbb{C}}
\newcommand{\N}{\mathbb{N}}
\newcommand{\R}{\mathbb{R}}
\newcommand{\cS}{\mathcal{S}}

\newcommand{\ep}{\varepsilon}
\newcommand{\pvo}{\text{\rm p.v.}~\tfrac{1}{x}}
\newcommand{\sgn}{\text{\rm sgn}}
\newcommand{\pvn}{\text{\rm p.v.}~\tfrac{1}{x^n}}
\newcommand{\pvno}{\text{\rm p.v.}~\tfrac{1}{x^{n+1}}}

\newtheorem{stat}{Statement}[section]
  \newtheorem{prop}[stat]{Proposition}
  
  \newtheorem{thm}[stat]{Theorem}
  
  \newtheorem{remark}[stat]{Remark}
  \newtheorem{def1}[stat]{Definition}

\begin{document}
\begin{center}
{ \scshape \bf \large A simple derivation of the Fourier transform of the Heaviside function}\footnote[1]{Version June 8, 2026. 



%
} 
\vspace{.5cm} 
	
	{\scshape {\large Robert C.~Dalang} \\
	 \'Ecole Polytechnique F\'ed\'erale de Lausanne} 
\end{center}

\thispagestyle{empty}

\begin{abstract}
We give a rigorous derivation of the Fourier transform of the Heaviside function within a framework for tempered distributions that is suitable for undergraduate engineering and mathematics students. The proofs rely on fundamental concepts typically taught in a freshman-level calculus course, including limits, generalized integrals, integration by parts and the Taylor Remainder Theorem. In passing, we examine the Principle Value of $\frac{1}{x}$ and the relationship between its derivative of order $n$ and the Principle Value of $\frac{1}{x^{n+1}}$.
\end{abstract}

\section{Introduction}\label{s1}
Engineering students interested in studying signal processing can benefit from a mathematics course that introduces the theory of tempered distributions in a user-friendly but rigorous way. The objective is to familiarize them with the Dirac delta function, the Fourier transform of tempered distributions related to this function, such as its derivative or antiderivative, and some related topics such as the Fourier transform of periodic functions. None of these questions can be dealt with within the more classical theory of Fourier transforms of functions in $L^1(\R)$ or $L^2(\R)$. 

In this note, we give a simple definition of the set of tempered distributions $\cS'(\R)$ and of the Fourier transform of an element $T \in \cS'(\R)$, with the goal of providing a straightforward derivation of the Fourier transform of the Heaviside function 
\begin{align}\label{rd06_08e1}
    H(x) := 1_{\R_+}(x) = \begin{cases} 1 & \text{ if } x \geq 0, \\
                0& \text{ if } x < 0.
                \end{cases}
\end{align}
Indeed, this function is one of the emblematic examples to which the Fourier transforms in the $L^1(\R)$ and $L^2(\R)$ contexts do not apply, along with $\cos(x)$ and $\sin(x)$, of course. Also, it plays an important role in the theory of Fourier transforms (see e.g.~\cite{bracewell}), and yet, while $\cF(H)$ is found in most standard references on this topic, usually via the Fourier transform of the {\em sign function} $\sgn$, the proof is often indirect or makes use of exercices that can be simplified. For instance, \cite{kammler} gives a short calculation for the Fourier transform of $\sgn$, but refers within the calculation to his Exercise 7.7 p.~457 where a nonstandard Fubini's theorem is proved. Other standard texts give proofs via regularization. And compared  to the Fourier transforms of $\sin(x)$ and $\cos(x)$, which are just linear combinations of Dirac delta functions, the Fourier transform of the Heaviside function has a much richer structure.

Indeed, the Fourier transform of the Heaviside function contains the Principle Value of the function $\frac{1}{x}$, denoted $\pvo$, so we begin by explaining how the formula for this tempered distribution is derived. In fact, many descriptions of the Principle Value of $\frac{1}{x}$ presented in standard references tend to be somewhat mysterious, whereas this is in fact very a simple object: indeed, $\frac{1}{x}$ is the second derivative, {\em in the sense of Newton and Leibnitz,} of the function $h(x) :=x \ln(\vert x \vert) - x$ ($x \in \R^* := \R \setminus \{ 0 \}$), whereas the Principle Value of $\frac{1}{x}$ is the second derivative of the same function, but {\em in the sense of distributions.} In particular, as opposed to continuously differentiable functions, $h(x) = x \ln(\vert x \vert) - x$ is a beautiful example of a {\em continuous and slowly growing function} (see Section \ref{s2}) whose derivatives, in the sense of Newton and Leibnitz and in the sense of distributions, exist, are nontrivial, and do {\em not} coincide. Though this is not needed for the Fourier transform of the Heaviside function, we also compare the higher-order derivatives of $h$ in the sense of distributions with the Principle Value of $\frac{1}{x^n}$, $n \in \N^* := \N \setminus\{0\}$, denoted $\pvn$.


After a quick introduction to the theory of tempered distributions, we derive in Section \ref{s3} the formula for the Principle Value of $\frac{1}{x}$, then for $\frac{1}{x^n}$ by induction ($n \geq 1$). In Section \ref{sec4}, we derive the Fourier transform of the Heaviside function. The main mathematical tools for deriving $\pvn$ are typically taught in a freshman-level calculus course, namely, limits, generalized integrals, integration by parts and the Taylor Remainder Theorem. In order to discuss the Fourier transform of the Heaviside function, some minimal knowledge of the $L^1(\R)$-theory of Fourier transforms is required, and Fubini's theorem is needed once.
\bigskip

\noindent{\sc Acknowledgement.} I wanted to present Theorem \ref{rd05_19t2} and its proof in my second year calculus class for engineering students at EPFL. This course covers in particular Fourier series and Fourier transforms of tempered distributions. I could not find a suitable published reference, and Claude (by Anthropic, that I prefer to other AI softwares for ethical reasons) was able to give me the statement, but only the types of proofs that I mention in the Introduction. Finally, I typed up the proof myself and showed it to Claude and asked it to comment. It said ``This is a complete and rigorous proof with no regularization of $\sgn$. I have not seen this exact proof in the sources I was trained on. ... Did you come up with it yourself?" Then I asked it to search the internet for instances of this proof and it replied ``I have now searched fairly thoroughly, and I cannot find this proof anywhere in the literature. ... It seems genuinely original. ... It might be worth writing it up for publication, perhaps as a short note in a journal ... '' So without Claude, I would have shown this proof to my students but I would not have written this paper!

\section{Quick introduction to tempered distributions}\label{s2}

We recall that an indefinitely differentiable function $\varphi: \R \to \R$ is said to have {\em rapid decrease} if for all $n, m \in \N$,
$$
   \lim_{x \to \pm \infty}  x^n \, \varphi^{(n)}(x) = 0.
$$
The set of all such functions is denoted $\cS(\R)$. A function $f: \R \to \R$ is said to be {\em continuous and slowly growing} (CSG) if there is $n \in \N$ such that
$$
   \lim_{x \to \pm \infty} \frac{f(x)}{x^n} = 0.
$$
The set $\cS'(\R)$ of {\em tempered distributions} is then the family of all functionals $T: \cS(\R) \to \R$ for which there exists a CSG-function $f$ and $n \in \N$ such that for all $\varphi \in \cS$,
\begin{align}\label{rd05_18e1}
    \langle T, \varphi \rangle = (-1)^n \int_{-\infty}^{+\infty} f(x)\,  \varphi^{(n)}(x) \, dx.
\end{align}
Here, $ \langle T, \varphi \rangle$ denotes the evaluation of $T$ at $\varphi$ (and could be denoted $T(\varphi)$), and $ \varphi^{(n)}$ denotes the $n$-th order derivative of $\varphi$. The most famous tempered distribution is certainly the Dirac delta function $\delta_0$, defined for $\varphi\in \cS$ by $\langle \delta_0, \varphi \rangle = \varphi(0)$, which is of the form \eqref{rd05_18e1} with $f(x) = \max(x, 0)$ and $n = 2$.  The tempered distribution associated with the Heaviside function $H$ of \eqref{rd06_08e1} is obtained with the same $f$ and $n=1$, which yields
$$
   \langle H, \varphi \rangle := \int_0^{+\infty} \varphi(x)\, dx, \qquad \varphi \in \cS.
$$

The above definition of $\cS'(\R)$ is equivalent to the classical definition (see \cite[Chapter VII, Théorème VI p.~239]{schwartz}), and it has the advantage that no notions of topology are needed to state it. A weakness, however, is that in order to verify that some linear functional $T$ on $\cS(\R) $ does indeed define a tempered distribution, it is necessary to exhibit a suitable function $f$ and an integer $n$ such that \eqref{rd05_18e1} is satisfied.

The derivative of a distribution $T \in \cS'(\R)$ is the distribution $\dot T$ defined by 
$$
   \langle \dot T, \varphi \rangle := -\langle T, \varphi' \rangle,
$$ 
and higher-order derivatives are defined by iteration. 

The Fourier transform of an integrable function $\varphi:\R \to \R$ is defined by
\begin{align}\label{rd05_21e1}
   \cF \varphi(x) := \int_{-\infty}^{+\infty} e^{-i  x y}\, \varphi(y)\, dy,
\end{align}
where $i \in \C$ is the complex number such that $i^2 = -1$, and the Fourier transform $\cF T$ of $T \in \cS'(\R)$ is defined by the formula
\begin{align}\label{rd05_19e3}
   \langle \cF T, \varphi \rangle :=  \langle T, \cF \varphi \rangle.
\end{align}
For this definition to be coherent, one must check that $ \cF \varphi \in \cS$ whenever $ \varphi \in \cS$, which is a straightforward calculus exercise. If one wants to check that $\cF T$ belongs to $ \cS'(\R)$ by using the definition \eqref{rd05_18e1}, then this involves some additional work (which I do not cover in my own classes for engineering students).

\section{The Principle Value of $\frac{1}{x^n}$}\label{s3}

Naively, one might want to define a distribution $S$ associated to the function $x \mapsto \frac{1}{x}$ by the formula 
$
     \langle S, \varphi \rangle = \int_{-\infty}^{+\infty}  \frac{1}{x }\, \varphi(x)\, dx.
$
However, when $\varphi(0) \neq 0$, this formula is meaningless because $x \mapsto \frac{1}{x}$ is not a locally integrable function. 

For $x \in \R^*$, set 
$$
    h(x) = x \ln(\vert x \vert) - x,
$$
and $h(0) = 0$. Then $h$ is CSG and for $x \neq 0$, $h''(x) = \frac{1}{x}$. This motivates the following definition.

\begin{def1}\label{def1}
 The {\em Principle Value of $\frac{1}{x}$}, denoted $\pvo$, is the tempered distribution defined by $\pvo := \ddot h$, where $\ddot h$ is the second derivative of $h$ {\em in the sense of distributions,} that is,
 \begin{align*}
     \langle \pvo, \varphi \rangle := \int_{-\infty}^{+\infty} [x \ln(\vert x \vert) - x]\,  \varphi''(x) \, dx.
 \end{align*}
 \end{def1}

We have the following explicit formulas for $\left\langle \pvo, \varphi \right\rangle$.

\begin{prop}\label{prop1}
 For $\varphi \in \cS(\R)$,
 \begin{align}\label{rd05_19e1}
    \left\langle \pvo, \varphi \right\rangle =  \int_{0}^{+\infty}\, \frac{\varphi(x) - \varphi(-x)}{x}\, dx = \int_{-\infty}^{+\infty}\, \frac{\varphi(x) - \varphi(0)\, 1_{[-1,1]}(x)}{x}\, dx.
\end{align}
\end{prop}

\begin{proof}
By definition,
%
\begin{align*}
     \left\langle \pvo, \varphi \right\rangle &= \lim_{\epsilon \downarrow 0} \left[ \int_{-\infty}^{-\epsilon} [x \ln(- x ) - x]\, \varphi''(x)\, dx + \int_{\epsilon}^{+\infty} [x \ln( x ) - x]\, \varphi''(x)\, dx \right] .
\end{align*}
Using the change of variables $y = -x$ in the first integral, then renaming the variable, we see that
\begin{align*}
    \left\langle \pvo, \varphi \right\rangle
       &=  \lim_{\epsilon \downarrow 0}  \int_{\epsilon}^{+\infty} [x \ln( x ) - x]\, [\varphi''(x) - \varphi''(-x)]\, dx.
\end{align*}
Integrate by parts to see that 
\begin{align*}
  \left\langle \pvo, \varphi \right\rangle &= \lim_{\epsilon \downarrow 0} \left[ - [\ep \ln(\ep) - \ep]\, [\varphi'(\ep) + \varphi'(-\ep)] - \int_{\epsilon}^{+\infty} \ln(x ) \, [\varphi'(x) + \varphi'(-x)] \, dx\right] .
\end{align*}
The product term vanishes in the limit, so we have
\begin{align*}
  \left\langle \pvo, \varphi \right\rangle &=  \lim_{\epsilon \downarrow 0}  \left[  - \int_{\epsilon}^{+\infty} \ln(x) \, [\varphi'(x) +  \varphi'(-x)] \, dx
\right] .
\end{align*}
Integrate again by parts and observe again that the product term vanishes in the limit, to find that
\begin{align}\nonumber
  \left\langle \pvo, \varphi \right\rangle &=  \lim_{\epsilon \downarrow 0}  \left[ \ln(\ep) [\varphi(\ep) -  \varphi(-\ep)] + \int_{\epsilon}^{+\infty}  \frac{1}{x }\, [\varphi(x) - \varphi(-x)]\, dx
\right]\\[15pt]
&=  \int_{0}^{+\infty} \frac{\varphi(x) - \varphi(-x)}{x}\, dx.
\label{rd05_19e2}
\end{align}
This establishes the first equality in \eqref{rd05_19e1}. Notice that the is no integrability problem at the origin in the integral \eqref{rd05_19e2}, because $\varphi$ is differentiable.

For the second equality, subtract and add $\varphi(0)\, 1_{[0,1]}(x)$ on the numerator in \eqref{rd05_19e2}, split the fraction, then the integral, into two terms and do the change of variables $y = -x$ in the second term. This gives the second equality in \eqref{rd05_19e1} and completes the proof of the proposition.
\end{proof}

We now discuss higher-order derivatives of $h$ in the sense of distributions, which are also derivatives of $\pvo$: for $n \in \N$,
\begin{align*}
     h^{\boldsymbol{\cdot} (n+2)} = \left(\pvo \right)^{\boldsymbol{\cdot} (n)}
\end{align*}
where the dot indicates the derivative in the sense of distributions (note that these higher-order derivatives are not needed in Section \ref{sec4}). For this, we use the following notation: for $n \in \N^*$ and for $\varphi \in \cS(\R)$, let
$$
     T_{n-1} \varphi(x) := \varphi(0) + \varphi'(0)\, x  + \cdots + \tfrac{ \varphi^{(n-1)}(0)}{(n-1)!}\, x^{n-1}
$$
be the $(n-1)^{\text{\scriptsize st}}$-order Taylor polynomial of $\varphi$ at $0$. 

We also recall that the Principle Value of $\frac{1}{x^n}$, denote $\pvn$, is the tempered distribution defined for $\varphi \in \cS$ (see \cite[Chapter 1]{GV}) by
\begin{align*}
   \left\langle \pvn, \varphi \right\rangle := \int_{-\infty}^{+\infty}\, \left[\frac{\varphi(x) - T_{n-1} \varphi(x) }{x^n}\, 1_{]0,1[}(\vert x \vert) 
       + \frac{\varphi(x) }{x^n}\, 1_{[1,\infty[}(\vert x \vert) \right]\, dx.
\end{align*}
 It is not immediately obvious that this formula defines a tempered distribution in the sense of \eqref{rd05_18e1}, but this will be an immediate consequence of Theorem \ref{prop2} below. 


The next theorem gives a more explicit formula for $\left(\pvo \right)^{\boldsymbol{\cdot} (n)}$. 

\begin{thm}\label{prop2}
For $n \in \N^*$ and $\varphi \in \cS(\R)$,
 \begin{align}\nonumber
 & \tfrac{(-1)^{n-1}}{(n-1)!}\,  \left\langle \left(\pvo \right)^{\boldsymbol{\cdot} (n-1)}, \varphi \right\rangle \\
 &\qquad\qquad\qquad = \int_{0}^{+\infty}\, \frac{\varphi(x) - T_{n-1}\varphi(x) + (-1)^n\, [\varphi(-x) - T_{n-1} \varphi(-x)]}{x^n}\, dx. 
 \label{rd05_20e1}
 \end{align}
Equivalently,
 \begin{align}    \label{rd05_29e1}
  \tfrac{(-1)^{n-1}}{(n-1)!}\, \left(\pvo \right)^{\boldsymbol{\cdot} (n-1)} =  \pvn - 1_{\{n \geq 2\}} \sum_{k=0}^{n-2} \tfrac{(-1)^{n} + (-1)^{k}}{k! \, (n - k -1)} \ \delta_0^{\, \boldsymbol{\cdot} (k)},
 \end{align}
where $\delta_0^{\, \boldsymbol{\cdot} (k)}$ denotes the $k$-th order derivative of the Dirac delta function $\delta_0$.
\end{thm}

\begin{remark}
Notice that for the integal in \eqref{rd05_20e1},  there is no integrability problem at the origin, by Taylor's Remainder Theorem applied to the $C^\infty$-function $\varphi$. 
The situation at $+\infty$ is slightly more subtle: according to the comment just after \eqref{rd05_29e3} below, $T_{n-1}\varphi(x)  + (-1)^n\,  T_{n-1} \varphi(-x)$ is a polynomial of degree $n-2$ at most.
\end{remark}

\begin{proof}[Proof of Theorem \ref{prop2}]
We begin with the formula \eqref{rd05_20e1}. For $n=1$, $- T_{n-1}\varphi(x) -  (-1)^n\, T_{n-1} \varphi(-x) = 0$, so this is the first formula of Proposition \ref{prop1}. We now fix $n \geq 1$, assume by induction that \eqref{rd05_20e1} holds for $n$ and we establish \eqref{rd05_20e1} for $n+1$. By definition, 
$$
     \tfrac{(-1)^{n}}{(n)!}\,  \left(\pvo \right)^{\boldsymbol{\cdot} (n+1)} = -\frac{1}{n} \left( \tfrac{(-1)^{n-1}}{(n-1)!}\,   \left(\pvo \right)^{\boldsymbol{\cdot} (n)} \right)^{\boldsymbol{\cdot} },
$$ 
where the dot indicates the derivative in the sense of distributions. Therefore, for $\varphi \in \cS(\R)$,
\begin{align*}
\tfrac{(-1)^{n}}{(n)!}\,  \left \langle  \left(\pvo \right)^{\boldsymbol{\cdot} (n+1)}, \varphi \right\rangle 
      &
      = \tfrac{1}{n}  \left \langle \tfrac{(-1)^{n-1}}{(n-1)!}\,   \left(\pvo \right), \varphi' \right\rangle.
\end{align*}
By the induction hypothesis, this is equal to
\begin{align}\label{rd05_20e2}
  \frac{1}{n} \int_{0}^{+\infty}\,\frac{\varphi'(x) - T_{n-1}\varphi'(x) + (-1)^n\, [\varphi'(-x) - T_{n-1} \varphi'(-x)]}{x^n} \, dx.
\end{align}
Notice that
$$
    T_{n-1} \varphi'(x) =  \varphi'(0) + \varphi''(0)\, x  + \cdots + \tfrac{ (\varphi')^{(n-1)}(0)}{(n-1)!}\, x^{n-1} = (T_n \varphi)'(x) .
$$
We can therefore integrate by parts in \eqref{rd05_20e2} to see that
\begin{align}\nonumber
\left \langle \pvno, \varphi \right\rangle &= \lim_{N \to +\infty} \left(\left[ \frac{1}{n} \frac{\varphi(x) - T_{n} \varphi(x) + (-1)^{n+1}\, [\varphi(-x) - T_n \varphi(-x)]}{x^n} \right]_{0}^{N} \right. \\
   &\qquad \qquad\qquad +  \int_{0}^{N}\,\left. \frac{\varphi(x) - T_{n} \varphi(x) + (-1)^{n+1} [\varphi(-x) - T_{n} \varphi(-x)]}{x^{n+1}} \, dx \right).
\label{rd05_20e3}
\end{align}
By the Taylor Remainder Theorem, there is no problem in either of these terms at $x=0$. The product term is equal to
\begin{align}\label{rd05_29e2}
    \frac{\varphi(N)  - (-1)^{n}\, \varphi(-N) - [T_{n} \varphi(N) - (-1)^{n}\, T_{n} \varphi(-N)]}{n\, N^n} 
       . 
\end{align}
The first two terms have limit $0$ as $N \to + \infty$. Notice that
\begin{align}\label{rd05_29e3}
  T_{n} \varphi(y) - (-1)^{n}\, T_{n} \varphi(-y) = \sum_{k=0}^{n} \frac{\varphi^{(k)}(0)}{k!} \, y^k \left(1 - (-1)^{n+k} \right).
\end{align}
In particular, the coefficient of the term of degree $n$ vanishes, so this is a polynomial of degree at most $n-1$, and the expression in \eqref{rd05_29e2} has limit $0$ as $n \to \infty$. It follows that \eqref{rd05_20e3} establishes \eqref{rd05_20e1} for $n+1$. This completes the proof of \eqref{rd05_20e1} for all $n \in \N^*$.

 Turning to \eqref{rd05_29e1}, we note that by the observation concerning the degree of the polynomial in \eqref{rd05_29e2}, if $n \geq 2$, then
 \begin{align*}
 T_{n-1} \varphi(y) + (-1)^{n-1}\, T_{n-1} \varphi(-y) = T_{n-2} \varphi(y) + (-1)^{n-1}\, T_{n-2} \varphi(-y) .
 \end{align*}
 Therefore,  if $n \geq 2$, then \eqref{rd05_20e1} can be written 
 \begin{align}\nonumber
     & \int_{0}^{1}\, \frac{\varphi(x) - T_{n-1}\varphi(x) + (-1)^n\, [\varphi(-x) - T_{n-1} \varphi(-x)]}{x^n}\, dx \\
     &\qquad+ \int_{1}^{+\infty}\, \frac{\varphi(x) - T_{n-2}\varphi(x) + (-1)^n\, [\varphi(-x) - T_{n-2} \varphi(-x)]}{x^n}\, dx.
 \label{rd05_29e4}
  \end{align}
By the Taylor Remainder Theorem, 
$
\frac{\varphi(x) - T_{n-1}\varphi(x) }{x^n}
$
is integrable at the origin, and because of the difference of degrees, $ \frac{T_{n-2}\varphi(x) }{x^n} $ is integrable at $+\infty$. 
We can therefore split the fraction and integral in each of the integrals in \eqref{rd05_29e4} into two integrals, do the change of variables $y = -x$ in the second integral, and relabel the integration variable and simplify to see that 
 \begin{align}
 \tfrac{(-1)^{n-1}}{(n-1)!}\,  \left\langle \left(\pvo \right)^{\boldsymbol{\cdot} (n)}, \varphi \right\rangle 
    &= \int_{-1}^{1} \frac{\varphi(x) - T_{n-1}\varphi(x) }{x^n}\, dx + \int_{\{\vert x \vert > 1 \}} \frac{\varphi(x) - T_{n-2}\varphi(x) }{x^n}\, dx.
 \label{rd05_30e1}
 \end{align}
 The last integral can be split into two, yielding
  \begin{align*}
   \int_{\{\vert x \vert > 1 \}} \frac{\varphi(x)}{x^n}\, dx -  \int_{\{\vert x \vert > 1 \}} \frac{T_{n-2}\varphi(x) }{x^n}\, dx.
  \end{align*}
 Clearly,
  \begin{align*}
  - \int_{\{\vert x \vert > 1 \}} \tfrac{T_{n-2}\varphi(x) }{x^n}\, dx &= \sum_{k=0}^{n-2}  \tfrac{\varphi^{(k)}(0)}{k!} \, \tfrac{1 - (-1)^{1+k-n} }{1 + k - n} =  \sum_{k=0}^{n-2} \tfrac{(-1)^k}{k!}\, \tfrac{1 - (-1)^{1+k-n}}{1 + k - n}\, \langle \delta_0^{\, \boldsymbol{\cdot} (k)}, \varphi \rangle \\
    &= - \sum_{k=0}^{n-2} \tfrac{(-1)^{n} + (-1)^{k}}{k! \, (n - k - 1)} \, \langle \delta_0^{\, \boldsymbol{\cdot} (k)}, \varphi \rangle.
\end{align*}
The terms of this sum for which $k$ and $n$ have opposite parity vanish. Together with \eqref{rd05_30e1}, this completes the proof of Theorem \ref{prop2}.
\end{proof}

\begin{remark}
It is not difficult to check that the correction term in \eqref{rd05_29e1} is equal to the $n^{\text{\scriptsize th}}$ derivative, in the sense of distributions, of the function
\begin{align*}
  - 1_{\{n \geq 2\}}  \sum_{k=0}^{n-2} \tfrac{(-1)^n + (-1)^k}{k!\, (n-k-1)!}\, x_+^{n-k-2},
\end{align*}
where $x_+ := \max(x, 0)$, and therefore, $\pvn$ defines a tempered distribution in the sense of \eqref{rd05_18e1}.
\end{remark}

\section{Fourier Transform of the Heaviside Function}\label{sec4}

The Heaviside function is the function $H: \R \to \R$ defined by 
$$
    H(x) = \begin{cases}  0 &\text{if } x < 0,  \\ 1 & \text{if } x \geq 0 . \end{cases}
$$
The next theorem gives the Fourier transform of $H$ in the sense of definition \eqref{rd05_19e3}. Notice that $h$ belongs neither to $L^1(\R)$ nor to $L^2(\R)$.

Let 
$$
    \sgn(x) = \begin{cases}  -1 &\text{if } x < 0,  \\ \ \ \, 1 & \text{if } x \geq 0 . \end{cases}
$$ 
Then $H = \frac12 (1+ \sgn)$ and we will begin by identifying $\cF(\sgn)$.

\begin{prop}\label{rd05_19t1}
$\cF(\sgn) = -2 i\,  \pvo\, $.
\end{prop}

\begin{proof} By definition, for $\varphi \in \cS$, 
\begin{align*}
    \langle \cF(\sgn), \varphi \rangle &= \langle \sgn, \cF \varphi \rangle = \int_{-\infty}^{+ \infty}  \sgn(x)\, \cF \varphi(x)\, dx =  - \int_{-\infty}^{0} \cF \varphi(x)\, dx +  \int_{0}^{+\infty} \cF \varphi(x)\, dx \\ 
    &=  \int_{0}^{+\infty} [\cF \varphi(x) - \cF \varphi(-x)]\, dx.
\end{align*}
 Use the definition \eqref{rd05_21e1} of the Fourier transform of the integrable function $\varphi$ to write
\begin{align*}
  \langle \cF(\sgn), \varphi \rangle  &=   \int_{0}^{+\infty} dx\, \int_{-\infty}^{+ \infty} dy\, \left[e^{-i x y} - e^{i x y}\right]\, \varphi(y).
 \end{align*} 
Use distributivity to split the integral into two integrals, and in the second integral, do the change of variables $z = -y$, and then relabel the integration variable to see that
\begin{align*}
  \langle \cF(\sgn), \varphi \rangle  &=  \int_{0}^{+\infty} dx\, \int_{-\infty}^{+ \infty} dy\, e^{-i x y}\, [\varphi(y) - \varphi(-y)]\\ 
    &=  \lim_{N \to +\infty} \int_{0}^{N} dx\, \int_{-\infty}^{+ \infty} dy\,  e^{-i x y}\, [\varphi(y) - \varphi(-y)] .
 \end{align*} 
 Since the assumptions of Fubini's theorem are satisfied, we can permute the two integrals to obtain
  \begin{align*}
  \langle \cF(\sgn), \varphi \rangle
    &= \lim_{N \to +\infty}  \int_{-\infty}^{+ \infty} dy\,[\varphi(y) - \varphi(-y)]  \int_{0}^{N} dx\, e^{-i x y} \\ 
    &=   \lim_{N \to +\infty}  \int_{-\infty}^{+ \infty} dy\, [\varphi(y) - \varphi(-y)] \left[\frac{e^{-i x y} }{-iy} \right]_{x = 0}^{x = N} \\ 
    &= \lim_{N \to +\infty}  \int_{-\infty}^{+ \infty} dy\, [\varphi(y) - \varphi(-y)]  \left[\frac{1}{iy} -  \frac{e^{- i N y}}{iy}  \right] .
 \end{align*} 
 The integral can be written as the difference of two integrals, and since the first one does not depend on $N$,
  \begin{align}
   \langle \cF(\sgn), \varphi \rangle
    &=    -i \int_{-\infty}^{+ \infty} dy\, \frac{\varphi(y) - \varphi(-y)}{y} \ + i\, \lim_{N \to +\infty}  \int_{-\infty}^{+ \infty} dy\, e^{- i N y}\,  \psi(y),
\label{rd05_19e4}
\end{align}
where 
$$
   \psi(y) := \begin{cases} \frac{\varphi(y) - \varphi(-y)}{y} &\qquad\text{if } y \neq 0, \\[5pt]
                      2\, \varphi'(0) &\qquad \text{if } y = 0.
                  \end{cases}
$$
Then $\psi$ is a continuous and $C^1$-function that is integrable over $\R$ (in fact, it belongs to $\cS(\R)$, but we do not need this), therefore the limit in \eqref{rd05_19e4}, which is $\lim_{N \to +\infty} \cF\psi(N)$, vanishes according to a well-known property of the Fourier transform.

Subtracting and adding $\varphi(0)$ on the numerator of the first integral in \eqref{rd05_19e4}, we see that
 \begin{align*}
  \langle \cF(\sgn), \varphi \rangle
    &=  - 2 i \int_{-\infty}^{+ \infty} dy\, \frac{\varphi(y) - \varphi(0)}{y} =  -2 i \, \left\langle \pvo, \varphi \right\rangle,
\end{align*}    
according to Proposition \ref{prop1}. This proves Proposition \ref{rd05_19t1}.
\end{proof}

\begin{thm}\label{rd05_19t2}
The Fourier transform $\cF H \in \cS'(\R)$ of the Heaviside function $H$ is given by
$$
   \cF H  = - i\,  \pvo + \pi\,  \delta_0 .
$$
\end{thm}

\begin{proof}
The conclusion follows from Proposition \ref{rd05_19t1} and the facts that $H = \frac12 (1 + \sgn)$ and $\cF(1) = 2 \pi\,\delta_0$, as is easily verified from the definition \eqref{rd05_19e3} and the Fourier inversion formula.
\end{proof}

\begin{remark}
As a consequence of Theorem \ref{rd05_19t2}, we see that $\cF H$ can be expressed as in \eqref{rd05_18e1} if we take $n=2$ and $f(x) = - i\, h(x) + \pi \max(x, 0)$. 
\end{remark}

\noindent{\sc Acknowledgement.} The author thanks Yuta Wakasugi (Hiroshima University) and, particularly, Boris Buffoni (\'Ecole Polytechnique Fédérale de Lausanne) for pointing out mistakes in the first version of this paper.

\vskip 16pt

{\bf Author information:}
\vskip 16pt

{\scshape
Robert C.~Dalang}

Institut de Mathématiques

\'Ecole Polytechnique Fédérale de Lausanne (EPFL)

CH-1015 Lausanne

Switzerland
\smallskip 

robert.dalang@epfl.ch


\begin{thebibliography}{99}
\bibitem{bracewell} Bracewell, R.N. {\em The Fourier transform and its applications.} McGraw-Hill Ser. Electr. Engrg. Circuits Systems. McGraw-Hill Book Co., New York, 1986.

\bibitem{GV} Gel'fand, I.M. \& Shilov, G.E. {\em Generalized Functions, Vol.~1:} Properties and operations. Reprint of the 1964 English translation. AMS Chelsea Publishing, Providence, RI, 2016.

\bibitem{kammler} Kammler, D. {\em A first course in Fourier analysis} (2nd ed.). Cambridge University Press, Cambridge, 2007.

\bibitem{schwartz} Schwartz, L. {\em Théorie des distributions.} Publications de l'Institut de Mathématique de l'Université de Strasbourg, IX-X. Hermann, Paris, 1966.

\end{thebibliography}
\end{document}